\documentclass[a4paper,11pt]{article}
\usepackage{amsmath,amssymb,amsfonts,mathrsfs,wasysym}
\usepackage{amsthm,vmargin}
\usepackage{paralist}
\usepackage{hyperref}
\usepackage{booktabs}
\usepackage{times}
\usepackage{graphicx}
\usepackage[numbers,sort&compress]{natbib}
\usepackage{tikz}
\usetikzlibrary{shapes,backgrounds}
\usepackage{setspace}
\graphicspath{{_figures/}}

\setpapersize{A4} 

\hypersetup{
    bookmarks=true,         
    unicode=false,          
    pdftoolbar=true,        
    pdfmenubar=true,        
    pdffitwindow=true,      
    pdftitle={My title},    
    pdfauthor={Author},     
    pdfsubject={Subject},   
    pdfnewwindow=true,      
    pdfkeywords={keywords}, 
    colorlinks=true,        
    linkcolor=blue,         
    citecolor=blue,         
    filecolor=blue,         
    urlcolor=blue           
}

\newtheorem{lemma}{Lemma}
\newtheorem{theorem}{Theorem}

\newcommand{\bX}{\mathbf X}

\newcommand{\e}{{\mathbf E}}
\newcommand{\p}[1]{{\mathbf P}\left(#1\right)}
\newcommand{\pc}[1]{{\mathbf P}(#1)}
\newcommand{\R}{{\mathbb R}}

\newcommand{\PROB}{\mathbf{P}}

\newcommand{\EXP}{\e}
\newcommand{\IND}[1]{\mathbf{1}_{\{ #1 \}}}
\newcommand{\defeq}{\stackrel{\rm def}{=}}

\newcommand{\cel}[1]{\ensuremath{\lceil #1 \rceil}}

\newcommand{\noi}{\noindent}

\date{\today}

\title{Connectivity of sparse Bluetooth networks}
\author{N.~Broutin\thanks{Projet RAP, Inria Paris-Rocquencourt} \and L.~Devroye\thanks{McGill University} \and G.~Lugosi\thanks{ICREA and Pompeu Fabra University. GL acknowledges support by the Spanish Ministry of Science and Technology 
grant MTM2012-37195.}}

\begin{document}


\maketitle

\begin{abstract}
Consider a random geometric graph defined on $n$ vertices uniformly distributed
in the $d$-dimensional unit torus. Two vertices are connected if their 
distance is less than a ``visibility radius'' $r_n$.
We consider {\sl Bluetooth networks} that are locally sparsified random geometric graphs. Each vertex selects $c$ of its neighbors in the random geometric
graph at random and connects only to the selected points. 
 We show that if the visibility radius is at least of the order of $n^{-(1-\delta)/d}$
for some $\delta > 0$, then a constant value of $c$ is sufficient for
the graph to be connected, with high probability. It suffices to take
$c \ge \sqrt{(1+\epsilon)/\delta} + K$ for any positive $\epsilon$ where $K$
is a constant depending on $d$ only. On the other hand, with $c\le \sqrt{(1-\epsilon)/\delta}$,
the graph is disconnected, with high probability.
\end{abstract}


\section{Introduction and results}

Consider the following model of random ``Bluetooth networks''.
Let $\bX=X_1,\ldots,X_n$ be independent, uniformly distributed random points in $[0,1]^d$, and denote the set of these points by $\bX=\{X_1,\ldots,X_n\}$.
Given a positive number $r_n>0$ (the so-called visibility radius),
 define the random geometric graph $G_n(r_n)$ with 
vertex set
$\bX$ in which two vertices  $X_i$ and $X_j$ are connected by an edge
if and only if $D(X_i,X_j)\le r_n$, where
$$D(x,y) = \left(\sum_{i=1}^d \min(|x_i-y_i|, 1-|x_i-y_i|)^2 \right)^{1/2}$$
is the Euclidean distance on the torus.

It is well known (\citet{Pen03}) that, for any $\epsilon>0$,
$$
\lim_{n\to\infty}
\p{G_n(r_n) \text{~is connected}} 
=
\left\{
\begin{array}{l l}
0 & \text{~if~} r_n\le (1-\epsilon)\left(\frac{\log n}{nv_d }\right)^{1/d}\\
1 & \text{~if~} r_n\ge (1+\epsilon)\left(\frac{\log n}{nv_d}\right)^{1/d},
\end{array}
\right.
$$
where $v_d$ is the volume of the 
Euclidean unit ball in $\R^d$.

Note that in order to guarantee that a random geometric graph is
connected (with high probability), the average degree in the graph
needs to be at least of the order of $\log n$, which makes the graph
too dense for some applications.
To deal with this issue, one may sparsify the graph. 
This can be done in a distributed way by selecting, for each vertex $X_i$,
randomly,  and independently a subset of $c_n$ edges adjacent to $X_i$, 
and then considering the subgraph containing the selected edges only. 
The selection is done without replacement and if a vertex has less than
$c_n$ neighbors in $G_n(r_n)$, then we take all of its neighbors.
The obtained random graph model, coined {\sl Bluetooth network} (or
{\sl irrigation graph})
 has been 
introduced and studied in \citet{FeMaPaPe04,DuJoHaPaSo05,DuHaMaPaPe07,CrNoPiPu09}, and \cite{PePiPu09}.

Formally, the random Bluetooth graph
$\Gamma_n=\Gamma_n(r_n,c_n)$ is obtained 
as a random sub-graph of $G_n(r_n)$ as follows. 
For every vertex $X_i\in \bX$, 
we pick randomly, without replacement,
$c_n$ edges, each adjacent to $X_i$ in $G_n(r_n)$. 
(If the degree of $X_i$ in $G_n$ is less than $c_n$, all edges adjacent to
$X_i$ are kept in $\Gamma_n$.)
We also denote by $\Gamma_n^+(r_n,c_n)$ the directed graph
obtained by placing a directed edge from $X_i$ to $X_j$ whenever $X_j$
is among the $c_n$ selected neighbors of $X_i$.

We study connectivity of $\Gamma_n(r_n,c_n)$ 
for large values of $n$. A property of the
graph holds {\sl with high probability (whp)} when
the probability that the property holds converges to one as $n\to \infty$.

When $r_n>\sqrt{d}/2$, the underlying random geometric graph $G_n(r_n)$
is the complete graph and $\Gamma_n(r_n,2)$ becomes the ``$2$-out'' 
random subgraph of the complete graph studied in \citet{FeFr1982a},
where it is shown that the graph is connected with high probability.
\citet{DuJoHaPaSo05} extended this result by showing that when 
$r_n=r>0$ is independent of $n$,
$\Gamma_n(r,2)$ is connected with high probability. 
When $r_n\to 0$ as $n\to\infty$, \citet{CrNoPiPu09}
proved that there exist constants $\gamma_1,\gamma_2$ such that
if $r_n \ge \gamma_1 (\log n/n)^{1/d}$ and $c_n \ge \gamma_2 \log(1/r_n)$,
then $\Gamma_n(r_n,c_n)$ is connected with high probability. 
\citet{BrDeFrLu2011a} 
 proved that when $r_n$ is just above the connectivity threshold for the underlying 
graph $G_n(r_n)$, that is, when $r_n\sim \gamma (\log n/n)^{1/d}$ 
for some sufficiently large $\gamma$, the connectivity threshold for the irrigation graphs is
$$c_n^\star:=\sqrt{\frac{2\log n}{\log\log n}}.$$
More precisely, for any $\epsilon \in (0,1)$, 
\begin{equation}\label{eq:threshold_conn}
\lim_{n\to\infty}\p{\Gamma_n(r_n,c_n)~\text{is connected}}
=
\left\{
\begin{array}{l l}
0 & \text{if~} c_n \le (1-\epsilon) c_n^\star\\
1 & \text{if~} c_n \ge (1+\epsilon) c_n^\star.
\end{array}
\right.
\end{equation}

This result shows that the simple distributed algorithm building the irrigation 
graph guarantees connectivity with high probability while reducing the
average degree to about $2 c_n^\star$, 
which is much less than the  average degree of $\Theta(\log n)$ for the initial random geometric graph. 
However, the average degree still grows with $n$, and therefore 
$\Gamma_n(r_n,c_n)$ is not genuinely sparse. 

\medskip
\noi\textsc{main results.}\ 
The main result of this note is that at the price of slightly increasing
the visibility radius to $r_n\sim n^{-(1-\delta)/d}$ for some $\delta >0$,
a constant number of connections per vertex suffices to achieve connectivity
with high probability.
The lower bound of \citet{BrDeFrLu2011a} states that for any 
$\epsilon \in (0,1)$ and $\lambda \in [1,\infty]$, if $\gamma>0$ is a sufficiently
large constant,
$r_n \ge \gamma \left(\frac{\log n}{n}\right)^{1/d}$, 
$\frac{\log nr_n^d}{\log\log n} \to \lambda$, and
\[
c_n\le \sqrt{(1-\epsilon)\left(\frac{\lambda}{\lambda-1/2}\right)\frac{\log n}{\log nr_n^d}},
\] 
then $\Gamma_n(r_n,c_n)$ contains an isolated $(c_n+1)$-clique
and therefore is disconnected whp. 

When $r_n\sim n^{-(1-\delta)/d}$, we have $\lambda=\infty$ and therefore
if $c_n \le \lfloor(1-\epsilon)/\sqrt{\delta}\rfloor$, then the random
graph $\Gamma_n(r_n,c_n)$ is disconnected whp.  This
bound may seem weak since this value of $c_n$ is just a constant, independent of
$n$.  However, we show here that this bound is essentially tight: We
prove that when $r_n= \Omega(n^{-(1-\delta)/d})$, for some $\delta>0$,
then $\Gamma_n(r_n,c_n)$ is connected whp whenever
$c_n$ is larger than a certain constant, which depends on $\delta$ and
$d$ only. 

\begin{theorem}\label{thm:connect}
Let $\delta \in (0,1)$, $\gamma >0$, and $\epsilon \in (0,1)$ be fixed.
Suppose that $r_n=\gamma_n n^{-(1-\delta)/d}$ where $\gamma_n/\gamma\to 1$
as $n \to \infty$. There exists a constant
$c$, depending on $\delta$ and $d$ only, such that the random Bluetooth graph
$\Gamma(r_n,c)$ is connected whp.
For $x\in(0,1)$ set 
\[
f(x):=\lceil\sqrt{(1+x^2+8\sqrt x+\epsilon)/(x- 2 x^2\log_2(1/x)}\rceil~.
\]
One may take $c=k_1+k_2+k_3+1$, where
\[
k_1=
\left\{
\begin{array}{l l}
\lceil f(\delta)\rceil & \text{if~}\delta\in(0,1/5)\\
\lceil f(1/5) \rceil & \text{if~}\delta\in(1/5,1)
\end{array}
\right.,
\quad
k_2=\left\lceil 
   \frac{8(1+\epsilon) v_d(2\sqrt{d})^d}{(1-\epsilon)}\right\rceil~,
\quad
k_3=\lceil \sqrt{4(1+\epsilon)v_d/\alpha_d}\rceil
\]
where $v_d$
is the volume of the Euclidean ball of radius $1$ in $\R^d$
and $\alpha_d=(1-\epsilon)/(2(2\sqrt{d})^d)$.
\end{theorem}

Furthermore, as $\delta$ becomes small, our upper bound for the constant
depends on $\delta$ as $((1+\epsilon)/\delta)^{1/2}+O(1)$, essentially matching the lower 
bound of \citet{BrDeFrLu2011a}  mentioned above. 
A straightforward combination of the lower bound and Theorem \ref{thm:connect}
implies the following:

\begin{theorem}\label{thm:small_delta}
Let $\overline{c}(\delta)$ denote the smallest integer $c$ for which 
$$\lim_{n\to\infty} \p{\Gamma_n(r_n,c)\text{ is connected}}=1$$  
and 
let $\underline{c}(\delta)$ denote the largest integer $c$ for which 
$$\lim_{n\to\infty} \p{\Gamma_n(r_n,c)\text{ is disconnected}}=1$$
when $r_n=\gamma_n n^{-(1-\delta)/d}$ with $\gamma_n\sim \gamma$ as 
$n\to\infty$. Then, for any $\epsilon>0$, we have, for all $\delta$ small enough,
$$(1-\epsilon) \delta^{-1/2} 
\le \underline{c}(\delta) \le \overline{c}(\delta) \le (1+\epsilon)\delta^{-1/2}.$$
\end{theorem}

Interestingly, the threshold is essentially independent of the value of $\gamma$
and the dimension $d$. This phenomenon was already observed in (\ref{eq:threshold_conn}).

\medskip
\noi\textsc{remarks and open questions.}\ 
Before we conclude this section, we mention a few questions that might be worth 
investigating. Theorem~\ref{thm:small_delta} above finds the correct asymptotics 
for the thresholds $\underline{c}(\delta),\overline{c}(\delta)$ when the visibility radius is 
$r_n\sim \gamma n^{-(1-\delta)/d}$ for small $\delta$. 
One may phrase a related question as follows: 
given a fixed a constant $c$, (a budget, in some sense) can one 
find a threshold function $r_n^\star=r_n^\star(c)$ for connectivity? 
More precisely, $r_n^\star$ 
should be such that, for any $\epsilon>0$,
$$\lim_{n\to\infty }\p{\Gamma_n(r_n, c)\text{ is connected}} 
= \left\{
\begin{array}{l l}
0 & \text{if } r_n<(1-\epsilon) r_n^\star\\
1 & \text{if } r_n>(1+\epsilon) r_n^\star.
\end{array}
\right.
$$
In some sense, Theorem~\ref{thm:small_delta} gives the asymptotics of the threshold
function for $c\to\infty$, but one would like to know the threshold for fixed values of 
the budget $c$. As it was proved in \cite{BrDeFrLu2011a} in the case that 
$r_n\sim \gamma (n^{-1}\log n)^{1/d}$, 
the main obstacles to connectivity should be isolated 
$(c+1)$-cliques. One could also try to prove that this is indeed the case at a finer 
level: around the threshold, the number of isolated $(c+1)$-cliques should be 
asymptotically Poisson distributed. Thus, one expects that the probability that 
$\Gamma_n(r_n,c)$ is connected should have asymptotics similar to that of 
classical random graphs \cite[][Theorem 7.3]{Bollobas2001} or  random 
geometric graphs \cite[][Theorem~13.11]{Pen03}, where the isolated vertices are the 
main obstacle.
Finally, we mention that elsewhere (\citet{BrDeLu2014b}) we investigate the birth of the 
giant component of $\Gamma(r_n,c)$.

\section{Proof sketch}

In the course of the proof, we  condition on the location of the points and assume 
that they are sufficiently regularly distributed. The probability that this happens is
estimated in the following lemma.
Let $N(A)=\sum_{i=1}^n \IND{X_i\in A}$ denote the number of points
falling in a set $A\subset [0,1]^d$ and let $\lambda$ denote the Lebesgue
measure. Denote by $B_{x,r}= \{y\in [0,1]^d: D(x,y) < r\}$ the open ball
centered at $x\in [0,1]^d$ and by 
$C_{x,r}= \{y: \forall i=1,\ldots,n, \min(|x_i-y_i|, 1-|x_i-y_i|) \le r/2\}$ 
the cube of side length $r$ centered at $x\in [0,1]^d$.

\begin{lemma}
\label{lem:regularity}
Suppose $r_n \ge \gamma n^{-(1-\delta)/d}$ for some
$\delta \in (0,1)$ and $\gamma >0$. 
Let $\epsilon>0$ and denote by $F$ the event that
for all $x\in [0,1]^d$,
\[
   \frac{N(B_{x,r_n})}{n\lambda(B_{x,r_n})} \in (1-\epsilon,1+\epsilon)
  \quad \text{and} \quad    
   \frac{N(C_{x,r_n/(2\sqrt{d})})}{n\lambda(C_{x,r_n/(2\sqrt{d})})} \in (1-\epsilon,1+\epsilon)~.
\]
Then there exists a constant $\theta =\theta(\delta,\epsilon)>0$ 
and a positive integer $n_0=n_0(\delta,\epsilon)$
such that
for all $n> n_0$,
$\pc{F}\ge 1-\exp(-\theta n r_n^d)$.
\end{lemma}

\begin{proof}
The lemma may be proved by standard arguments. It follows, for example,
from inequalities for uniform deviations of empirical measures over 
{\sc vc} classes such as Theorem~7 in \cite{BoBoLu2004}.
\end{proof}

%
%

In the rest of the proof, we fix $\bX=\{X_1,\ldots,X_n\}$, and
assume that event $F$ holds. Thus, all randomness originates from
the random choices of the $c$ neighbors of every vertex. 
We denote by $\PROB_c$ the probability with respect to the random
choice of the neighbors only (i.e., conditional given the set $\bX$).
It suffices to show that if $F$ holds, then with high
probability (with respect to $\PROB_c$), $\Gamma(r_n,c)$ is connected.

\medskip
\noi\textsc{sketch of the proof.}\
The general strategy is to prove that, with high probability,
from any two points $X_i$, $X_j$, one can find a path that connects 
$X_i$ to $X_j$. To do this, most of the work consists in proving structural properties 
of the connected component containing a fixed point $X_1$. 
We study the random graph by dividing the $c$ neighbors into 
four disjoint groups of sizes $k_1,k_2$, $k_3$, and $1$, thus obtaining
four independent sets of edges added in four different phases.
The sketch of the proof of Theorem~\ref{thm:connect} is as follows. 
We rely on a discretization of the unit cube $[0,1]^d$ into congruent 
cubes of side length $1/\cel{2r_n\sqrt d}$. 

\begin{itemize}
\item \textsc{searching for a dense cube I.} (Section~\ref{sec:first-growth}).\ 
In the first phase we start from an arbitrary vertex, say, $X_1$, and
using only $k_1$ choices of each vertex, consider the set of the vertices
of $\bX$ which may be reached using paths of at most 
$\ell \approx \delta^2 \log_2 n$ edges. 
We show that if this growth process succeeds, there exists a cube in the grid 
partition that contains a connected component of size at least
$n^{\min\{\delta^2,1/25\}/3}$. 

\item \textsc{searching for a dense cube II.} (Section~\ref{sec:second-growth}).\ 
In the second phase we show that by adding $k_2$ new connections
to each vertex in the component obtained after the first step,
at least one of the grid cells has a positive fraction of its points in a single 
connected component. 

\item \textsc{propagating the density.} (Section~\ref{sec:third-growth}).\ 
Once we have a cell containing a constant proportion of points belonging 
to the same connected component, it is rather easy to propagate this positive 
density (of a single connected component) to all other cells of the grid by 
using $k_3$ new connections per vertex. 

\item \textsc{connectivity is unavoidable.} (Section~\ref{sec:final-growth}).\ 
The previous phases guarantee that from a single vertex $X_1$ all three phases
succeed with probability $1-o(1/n)$. So, with high probability,
the connected components of {\sl every} single vertex $X_i$, $1\le i\le n$, 
reach in every corner of the space. Then, it is easy to show that with just 
one additional connection per vertex, any two such components very likely connect, 
proving that the entire graph is, in fact, connected.
\end{itemize}

\section{First growth process: searching for a dense cube I}
\label{sec:first-growth}

Divide $[0,1]^d$ into a grid of congruent cubes of side length
$1/\cel{2\sqrt{d}/r_n}$. We first prove that a constant number of edges per vertex 
suffices to guarantee that there exists a cell that contains at least a polynomial 
number of points of $\bX$ that may be reached from $X_1$. 

\begin{lemma}\label{lem:first}
Suppose $\bX=\{X_1,\ldots,X_n\}$ are such that the event $F$ defined in Lemma \ref{lem:regularity} occurs. 
Let $k_1$ and $r_n$ be defined as in Theorem~\ref{thm:connect}.
With probability at least $1-  o(1/n)$,
the connected component of $\Gamma(r_n,k_1)$ containing $X_1$ is such that
there exists a cell in the grid partition of $[0,1]^d$ into 
congruent cubes of side length $1/\lceil 2\sqrt{d}/r_n\rceil$
that contains 
at least $n^{\min\{\delta^2,1/25\}/3}$ vertices of the component.
\end{lemma}

\begin{proof}
Let $\ell< n$ be a positive integer specified below.
Consider $A$, the set of vertices of $\bX$ that can be reached from $X_1$ using 
a directed path of length at most $\ell$ in $\Gamma_n^+(r_n,k_1)$. 
Note that $|A|\le 1+k_1+k_1^2+\cdots +k_1^{\ell}$. We first show 
a lower bound on the size of this connected component:
\begin{align}\label{eq:bound_branching}
\pc{|A|\le k_12^{\ell-1}} = o(1/n).
\end{align}
To see this, the key property is that with high probability, 
the number of new points added in the second generation is at least $2k_1$. 
First, the $k_1$ points of the first generation are distinct for they are 
sampled without replacement. 
For the second generation, imagine $k_1^2$ bins, $k_1$ for each of the
$k_1$ vertices of the first generation, into which we place the points
chosen by these vertices.
For these $k_1^2$ bins to contain only $j$ different points that are 
also different from the points of the first generation, there must
exist $k_1^2-j$ bins that contain only points of the $j$ remaining bins
or points from the first generation. 
There are $\binom{k_1^2}{j}$ ways of choosing these $k_1^2-j$ bins
and for each such bin, the probability that it contains a point either from
the remaining $j$ bins or from points of the first generation is at most
$(1+k_1+k_1^2)/(nr_n^d v_d(1-\epsilon)-1-k_1-k_1^2)$
(since, on the event $F$, each ball of radius $r_n$
contains at least $nr_n^d v_d(1-\epsilon)$ points).
Thus, the probability that
the number of distinct points 
 in the second generation that are distinct and do not belong to the first generation is less than $2k_1$ is at most
\[
 \sum_{j=0}^{2k_1} \binom{k_1^2}{j} 
  \left(\frac{1+k_1+k_1^2}{nr_n^d v_d(1-\epsilon)-1-k_1-k_1^2}\right)^{k_1^2-j}~.
\]

Similarly, assuming that there are at least $2k_1$ points in the second generation, the probability that the number of selected
neighbors in the third generation not selected before is less than
$4k_1$ is at most
\[
 \sum_{j=0}^{4k_1} \binom{2k_1^2}{j} 
  \left(\frac{1+k_1+k_1^2+k_1^3}{nr_n^d v_d(1-\epsilon)-1-k_1-k_1^2-k_1^3}\right)^{2k_1^2-j}~.
\] 
We may continue in this fashion for $\ell-1$ steps, in each step 
doubling the number of neighbors with high probability.
The probability that the $\ell$-th generation has less than $2^{\ell-1}k_1$
vertices is at most
\begin{align}
& \sum_{i=1}^{\ell-1}  \sum_{j=0}^{2^{i-1}k_1} \binom{2^{i-1}k_1^2}{j} 
  \left(\frac{1+k_1+\cdots+k_1^i}{nr_n^d v_d(1-\epsilon)-(1+k_1+\cdots+k_1^i)}\right)^{2^{i-1}k_1^2-j}\notag
\\
& \label{eq:bound_c1}
\quad \le 
 \sum_{i=1}^{\ell-1}  2^{i-1}k_1 2^{2^{i-1}k_1^2} 
  \left(\frac{k_1^{\ell}}{nr_n^d v_d(1-\epsilon)-k_1^{\ell}}\right)^{2^{i-1}k_1^2-2^ik_1}.
\end{align}
Now, we choose $\ell = \ell(\delta) =\lfloor \min\{\delta^2,1/25\} \log_2 n \rfloor$, 
and distinguish two cases depending on the value of~$\delta$. 

\noi (i) Suppose first that $\delta\in (0,1/5]$. In this case, we aim at obtaining
a value for $k_1$ that matches the lower bound. Then, in this range,
\begin{align*}
\log_2 k_1 \le \log_2 \left(\sqrt{\frac{1+\delta^2 + 8 \sqrt \delta + \epsilon} 
{\delta-2\delta^2 \log_2(1/\delta)}}+1\right)
\le 2 \log_2 (1/\delta),
\end{align*}
for any $\epsilon\in (0,1)$ and $\delta\in(0,1)$. Thus, we have 
$k_1^\ell \le n^{\delta^2 \log_2 k_1} \le n^{2 \delta^2 \log_2 (1/\delta)}.$
It follows that the right-hand side in \eqref{eq:bound_c1} above is at most 
\begin{align*}
& \sum_{i=1}^{\ell-1}  2^{i-1}k_1 
  \left(2^{1/(1-2/k_1)}\right)^{2^{i-1} k_1^2 (1-2/k_1)} 
  \left(\frac{n^{2 \delta^2 \log_2 (1/\delta)}}
  {\gamma_n^d n^{\delta} v_d(1-\epsilon)-
  n^{2\delta^2 \log_2(1/\delta) }}\right)^{2^{i-1}k_1^2(1-2/k_1)} \\
& \quad \le \sum_{i=1}^{\ell-1}  2^{i-1}k_1 
  \left(\frac{2^{1/(1-2/k_1)}}{\gamma_n^d v_d (1-2\epsilon)} 
  n^{-\delta + 2 \delta^2 \log_2 (1/\delta)} \right)^{2^{i-1}k_1^2(1-2/k_1)}\\
& \quad \le \ell 2^\ell k_1 \kappa_d^{k_1^2} 
n^{-k_1^2(1-2/k_1) (\delta-2\delta^2 \log_2(1/\delta) )}\\
& \quad \le k_1 \kappa_d^{k_1^2} \cdot n^{\delta^2} \log_2(n) \cdot 
n^{-k_1^2(1-2/k_1) (\delta-2\delta^2 \log_2(1/\delta))}
\label{eq:bound_c1f}
\end{align*}
for $n$ sufficiently large, where $\kappa_d=8/(\gamma^dv_d(1-2\epsilon))$.
Now, by our choice of $k_1$, we have 
\begin{align*}
&\delta^2 - k_1^2(1-2/k_1) (\delta - 2\delta^2 \log_2(1/\delta))\\
&\le \delta^2 - (1+\delta^2+8\sqrt \delta +\epsilon) + 
2 \sqrt{(1+\delta^2+8\sqrt \delta +\epsilon)(\delta-2\delta^2 \log_2(1/\delta))}\\
& \le 1-8\sqrt \delta - \epsilon + 2 \sqrt{12 \delta}\\
& \le 1-\epsilon,
\end{align*}
hence the probability in \eqref{eq:bound_c1} above is $o(1/n)$ and the bound in 
\eqref{eq:bound_branching} is proved for $\delta\in (0,1/5]$.

\medskip
\noi (ii) Suppose next that $\delta\in (1/5,1)$. 
In this case, the bound follows trivially from the fact that the case $\delta=1/5$ is 
covered by case (i). Indeed, we have
\begin{align*}
&\sum_{i=1}^{\ell-1}  2^{i-1}k_1 2^{2^{i-1}k_1^2} 
  \left(\frac{k_1^{\ell}}{nr_n^d v_d(1-\epsilon)-k_1^{\ell}}\right)^{2^{i-1}k_1^2-2^ik_1}\\
&\quad \le \sum_{i=1}^{\ell-1}  2^{i-1}k_1
\left(\frac{4k_1^{\ell(1/5)}}{nr_n^d v_d(1-\epsilon)-k_1^{\ell(1/5)}}\right)^{2^{i-1}k_1^2(1-2/k_1)}.
\end{align*}
It follows that, in this range also, the probability in \eqref{eq:bound_c1}
is $o(1/n)$ so that the bound \eqref{eq:bound_branching} is proved for $\delta \in (1/5,1)$.

Thus, we have shown that if $n$ is sufficiently large, then
with probability at least $1-o(1/n)$, $A$ contains at least 
$k_12^{\ell-1} \ge n^{\min\{\delta^2,1/25\}/2}$ vertices,
all within distance $\ell r_n \le \min\{\delta^2,1/25\} \log_2 n \cdot \gamma_n n^{-(1-\delta)/d}$
from $X_1$. 
In particular, the points of $A$ all fall in grid cells at most 
$\lceil r_n \ell \rceil /(r_n/(2\sqrt{d})) +1 \le 1+ 2\ell\sqrt{d}$
away from the cell containing $X_1$. 
Thus, all these vertices fall in a cube of at most $(3+4\ell\sqrt{d})^d$ cells. 
This implies that 
there must exist a cell with at least
\[
\frac{n^{\min\{\delta^2,1/25\}/2}}
{(3+4\ell\sqrt{d})^d} \ge n^{\min\{\delta^2,1/25\}/3}
\]
vertices for $n$ large enough.
\end{proof}

\section{Second growth process: searching for a dense cube II}
\label{sec:second-growth}

We now show that we can leverage Lemma~\ref{lem:first} and obtain,
still using a constant number of extra edges per vertex, 
a cell that contains a positive density of points of the connected component 
containing $X_1$. 

\begin{lemma}\label{lem:second}
Suppose event $F$ occurs.
Let $k_1,k_2$, and $r_n$ be defined as in Theorem \ref{thm:connect}.
With probability at least $1- o(1/n)$,
the connected component of $\Gamma(r_n,k_1+k_2)$ containing $X_1$ is such that
there exists a cell in the grid partition of $[0,1]^d$ into 
congruent cubes of side length $1/\lceil 2\sqrt{d}/r_n\rceil$ that contains 
at least $\alpha_d nr_n^d$  vertices of the component, where
$\alpha_d=(1-\epsilon)/(2(2\sqrt{d})^d)$.
\end{lemma}

By Lemma \ref{lem:first} we see that,
with probability at least $1- o(1/n)$,
 after $k_1$ connections per vertex,
there exists a grid cell that contains at least $n^{\min\{\delta^2,1/25\}/3}$ vertices
of the connected component containing $X_1$. Next we show that
with 
\[
k_2=\left\lceil 
   \frac{8(1+\epsilon) v_d(2\sqrt{d})^d}{(1-\epsilon)}\right\rceil
\]
new connections, the same cell contains a constant
times $nr_n^d$ vertices in the same connected component (i.e.,
a linear fraction of all points in the cell).

Consider the cell that contains the largest number of vertices in the
connected component containing $X_1$ after the first $k_1$ connections,
and let $N_0$ denote the number of such connected vertices in this cell.  
Then we have seen in Section~\ref{sec:first-growth} that
$$
\pc{N_0<  n^{\min\{\delta^2,1/25\}/3}} =o(1/n)
$$
with respect to the random choices of the first $k_1$ connections.
Suppose that the event $N_0\ge  n^{\min\{\delta^2,1/25\}/3}$ holds.

Now add $k_2$ fresh connections to each of these $N_0$ points, resulting in
$N_1\le k_2N_0$ vertices in the same grid cell that haven't been in the 
connected component of $X_1$ so far. 
If the number $N_1$ of new vertices in the cell added to the
component is less than
$2n^{\min\{\delta^2,1/25\}/3}$, we declare failure, otherwise continue by adding 
$k_2$ new 
connections to these $N_1$ points. (Note that since these $N_1$ vertices
did not belong to the component of $X_1$ before the first step, 
we have not discovered any of their connections and we may use 
$k_2$ new connections per vertex.) In this step we add $N_2 \le k_2N_1$
new vertices in the same grid cell. If $N_2 < 4n^{\min\{\delta^2,1/25\}/3}$, 
we declare failure,
otherwise continue. We repeat adding $k_2$ connections to all newly discovered
vertices until the number of connected vertices in the grid cell
$N_0+N_1+\ldots + N_i$ reaches $\alpha_dnr_n^d$ where
$\alpha_d=(1-\epsilon)/(2(2\sqrt{d})^d)$
or else
 for $L$ steps, requiring in every step $i=1,\ldots,L$ that
the number of newly discovered vertices in the same cell be at least
$2^in^{\min\{\delta^2,1/25\}/3}$. Here $L$ is chosen such that
\[
2^Ln^{\min\{\delta^2,1/25\}/3} 
\le \alpha_d nr_n^d 
< 2^{L+1}n^{\min\{\delta^2,1/25\}/3}.
\]
To estimate the probability that the process described above fails, 
observe first that at step $i$, the (conditional) probability 
 that a 
vertex selects a neighbor in the same cell is at least
\[
   \frac{2\alpha_dnr_n^d - \sum_{j=0}^{i-1}N_i}
   {(1+\epsilon)nr_n^d v_d}
\ge p_d \defeq
   \frac{\alpha_d}
   {(1+\epsilon) v_d}
\]
since on the event $F$, every grid cell has at least
$(1-\epsilon)nr_n^d/(2\sqrt{d})^d=2\alpha_dnr_n^d$ 
points and the vertex can reach at
most $(1+\epsilon)nr_n^d v_d$ points. 
In the inequality we used the fact that the number
of vertices discovered until step $i$ during the process
$\sum_{j=0}^{i-1}N_i \le (1-\epsilon)nr_n^d/(2(2\sqrt{d})^d)$
otherwise the process stops with success before step $i$.

Thus, after $i-1$ successful steps,
the expected number of newly discovered vertices at stage $i$
is at least
\[
\EXP_i N_i \ge   N_{i-1}k_2 p_d \ge 2^{i-1} n^{\min\{\delta^2,1/25\}/3}k_2 p_d
\ge 2^{i+1} n^{\min\{\delta^2,1/25\}/3}
\]
by the definition of $k_2$,
where $\EXP_i$ denotes conditional expectation given the first 
growth process and the first $i-1$ steps of the second process.
Given that the process has not failed up to step $i-1$, the conditional
probability that it fails at step $i$ is thus
\begin{align*}
  \PROB_i \big( N_i < 2^{i+1} n^{\min\{\delta^2,1/25\}/3} \big)
& \le 
  \PROB_i \big( N_i < \EXP_i N_i - 2^i n^{\min\{\delta^2,1/25\}/3} \big) \\
& \le  \exp\left(\frac{- k_2N_{i-1}p_d^2}{8}\right)\\
& \le \exp\left(\frac{- k_22^{i-1}n^{\min\{\delta^2,1/25\}/3}p_d^2}{8}\right)~,
\end{align*}
where we used the fact that conditionally, $N_i$ has a hypergeometric
distribution whose moment generating function is dominated by
that of the corresponding binomial distribution $B(k_2N_{i-1},p_d)$
(see Hoeffding \cite{Hoe63}) and we used simple Chernoff bounding
for the binomial distribution.

Thus, the probability that the process ever fails is bounded by
\begin{align*}
\sum_{i=1}^L \exp\left(\frac{- k_22^{i-1}n^{\min\{\delta^2,1/25\}/3}p_d^2}{8}\right)
& \le L\exp\left(\frac{- k_2n^{\min\{\delta^2,1/25\}/3}p_d^2}{8}\right)\\
& \le \exp\left(- n^{\min\{\delta^2,1/25\}/4}\right)
\end{align*}
for all sufficiently large $n$.

\section{Third growth process: propagating the density}
\label{sec:third-growth}

In this third step we show that, by adding a few more connections,
the connected component containing $X_1$ contains a linear fraction
of the points in {\sl every} cell of the grid partition of
$[0,1]^d$ into cubes of side length $1/\lceil 2\sqrt{d}/r_n\rceil$. 
In order to do so,
we start from the cell containing $\alpha_d nr_n^d$ vertices in
the connected component of $\Gamma(r_n,k_1+k_2)$ containing $X_1$,
and ``grow'' the component cell-by-cell until every cell has the required
number of vertices in the same connected component. We show that a constant
number of additional connections per vertex suffices.

\begin{lemma}\label{lem:third}
Suppose event $F$ defined in Lemma \ref{lem:regularity} occurs.
Let $k_1,k_2,k_3$, and $r_n$ be defined as in Theorem \ref{thm:connect}.
With probability at least $1- o(1/n)$,
the connected component of $\Gamma(r_n,k_1+k_2+k_3)$ 
containing $X_1$ is such that
every cell in the grid partition of $[0,1]^d$ into 
congruent cubes of side length $1/\lceil 2\sqrt{d}/r_n\rceil$ contains 
at least $2\alpha_d nr_n^d/(3k_3)$ vertices of the component.
\end{lemma}
\begin{proof}
Suppose the event described in Lemma \ref{lem:second} holds so that
there exists a cell with $\alpha_d nr_n^d$ vertices in the
connected component of $X_1$. Label this cell by $1$.
Next label all the cells from $1$ to $\lceil 2\sqrt{d}/r_n\rceil^d$
in such a way that cells $i$ and $i+1$ are adjacent (i.e., they share a
$d-1$-dimensional face), for all $i=1,\ldots,\lceil 2\sqrt{d}/r_n\rceil^d-1$.
Note that the size of the cells was chosen such that for any $x$ in cell $i$ 
the ball $B(x,r)$ contains entirely cell $i+1$. In particular, every point 
in cell $i+1$ is a potential neighbor of a point in cell $i$.

Let $k_3=\lceil \sqrt{4(1+\epsilon)v_d/\alpha_d}\rceil$ and 
consider the following process. Select, arbitrarily, 
$$M := \left\lfloor\frac{2\alpha_d nr_n^d}{3k_3}\right\rfloor$$ 
of the already connected vertices in cell $1$ and select, one-by-one, 
$k_3$ new neighbors
for each until $M$ newly connected
vertices have been found in cell $2$.
Declare failure if the number of newly connected vertices
in cell $2$ is less than $M$ after revealing all possible $k_3M$
new  connections.
Otherwise continue
and select $k_3$ new neighbors of the new vertices in cell $2$. 
Visit all cells in a sequential fashion, always stopping 
the selection when $M$ new vertices are connected in 
the next cell. If the process succeeds, then $X_1$ is connected to 
at least $M$ points in every cell, which is a positive proportion.

To analyze the probability of failure, note that for 
each new connection, the probability of discovering a new vertex
in cell $i+1$ by a vertex in cell $i$ is at least
\[
\frac{2\alpha_dnr_n^d - k_3M}{(1+\epsilon)v_d nr_n^d}
\]
since (assuming the event $F$) there are at least $2\alpha_dnr_n^d$
vertices in cell $i+1$, at most $(1+\epsilon)v_d nr_n^d$ within
radius $r_n$ of the point whose neighbors we select, and we discard
at most $k_3M$ points cell $i+1$ that may have been 
already be chosen by other vertices in cell $i$.

Thus, the expected number
of newly discovered vertices in cell $i+1$, conditionally on the
fact that the process has not failed earlier, is at least
\[
  k_3\alpha_d nr_n^d \cdot 
  \frac{2\alpha_dnr_n^d - k_3M}{(1+\epsilon)v_d nr_n^d}
  \ge 
  k_3\alpha_d nr_n^d \frac{\alpha_d}{(1+\epsilon)v_d}
 \ge 2M~,
\]
for all $n$ large enough, where the last inequality follows from our choice of $k_3$.
Thus, the probability of failure in step $i$ is, by a similar
argument as in the proof of Lemma \ref{lem:second}, bounded by $e^{-\eta_d nr_n^d}$, 
where
\[
\eta_d = \frac{\alpha_d^3}{12k_3(1+\epsilon)^2 v_d^2}~.
\]
Hence, the probability that the process ever fails in any step is $o(1/n)$ and this 
completes the proof.
\end{proof}

\section{Final step: proof of Theorem~\ref{thm:connect}}\label{sec:final-growth}

With a giant component densely populating every cell of the grid, it is now
easy to show that with at most one extra connection,
the entire graph becomes connected, with high probability.
The argument is as follows.

Start with growing the component of $X_1$ as in the growth processes
described above. Then, with $k_1+k_2+k_3$ connections per vertex,
the component has the property described in Lemma~\ref{lem:third},
with probability $1-o(1/n)$. Now consider points $X_2,X_3,\ldots,X_n$,
one-by-one. If $X_2$ does not belong to the connected component of $X_1$,
then grow the same process starting from $X_2$ until the component
of $X_2$ hits the one of $X_1$. If the two components do not connect, then 
this new component also satisfies the property of Lemma \ref{lem:third}, with
probability $1-o(1/n)$.
(Note that until the two components meet, all connections are new in the
sense that they have not been revealed in the first process, and therefore
these events hold independently.) Now we may continue, taking  
every $X_i$, $i=3,\ldots,n$. If $X_i$ does not belong to any of the 
previously grown components, then we grow a new component starting from $X_i$
until it hits one of the previous ones, or else until the component
has at least $2\alpha_d nr_n^d/(3k_3)$ vertices in every cell of the grid.
By the union bound, with probability $1-o(1)$, every vertex is contained
in such a giant component. 

If there is more than one connected component at the end of the process, then 
we add one more connection to every vertex not belonging to the
component of $X_1$. Note that each such component has at least 
$n/(3k_3)$ vertices and for every vertex, the probability that it hits
the component of $X_1$ is at least $1/(3k_3)$, so the probability
that all new connections miss the component of $X_1$ is at most
$(3k_3)^{-n/3k_3}$. Thus, by the union bound, with high probability, 
all components connect to the first one and the entire graph
becomes connected. This concludes the proof of Theorem \ref{thm:connect}.

\section*{Acknowledgement}
Part of this work has been conducted while attending a workshop at the Banff International Research Station for the workshop on Models of Sparse Graphs and Network Algorithms. 

\setlength{\bibsep}{.3em}
\bibliographystyle{plainnat}
\bibliography{connectivity}

\end{document}